\documentclass[10pt,a4paper]{amsart}
\usepackage[foot]{amsaddr}
\usepackage{amsmath}
\usepackage{amsthm}
\usepackage{amssymb}
\usepackage{mathrsfs}
\usepackage{verbatim}
\usepackage{xcolor}
\usepackage[colorlinks=true,linkcolor=blue,citecolor=red]{hyperref}
\usepackage{enumitem}
\usepackage{multicol}
\usepackage{graphicx}
\usepackage{float}
\usepackage[utf8]{inputenc}
\usepackage[nameinlink]{cleveref}
\usepackage{array}

\theoremstyle{plain}
\newtheorem{theorem}{Theorem}
\newtheorem*{theorem*}{Theorem}
\newtheorem{lemma}[theorem]{Lemma}

\theoremstyle{definition}
\newtheorem{definition}[theorem]{Definition}

\theoremstyle{remark}
\newtheorem{remark}[theorem]{Remark}

\newcommand*{\avint}{\mathop{\ooalign{$\int$\cr$-$}}}

\newcommand{\RR}{\mathbb{R}}

\newcommand{\NN}{\mathbb{N}}

\newcommand{\D}{\Delta}
\newcommand{\test}{C_c^{\infty}}

\newcommand{\Om}{\Omega}
\newcommand{\II}{\iint_}
\newcommand{\I}{\int_}

\newcommand{\del}{\partial}

\numberwithin{equation}{section}
\numberwithin{theorem}{section}
\allowdisplaybreaks


\begin{document}
	\title[Bal-Mishra-Mohanta]{On the singular problem involving $g$-Laplacian}
	
	\author{$^1$Kaushik Bal}
	\email{$^1$kaushik@iitk.ac.in}
	\address{$^1$Indian Institute of Technology Kanpur, India}
	\author{$^2$Riddhi Mishra}
	\email{$^2$riddhi.r.mishra@jyu.fi}
	
	\author{$^3$Kaushik Mohanta}
	\email{$^3$kaushik.k.mohanta@jyu.fi}
	\address{$^{2,3}$University of Jyv\"askyl\"a, Finland}

	.
	\smallskip

	\begin{abstract}
		In this paper, we show that the existence of weak solution to the equation
		\begin{align*}
			\begin{split}
				(-\D_g)^su(x)&=f(x)u(x)^{-q(x)}\;\mbox{in}\; \Om,\\
				u&>0\;\mbox{in }\; \Om,\\
				u&=0\;\mbox{in}\;\RR^N\setminus\Om
			\end{split}
		\end{align*}
		where $\Om$ is a smooth bounded domain in $\RR^N$, $q\in C^1(\overline{\Om})$, and $(-\D_g)^s$ is the fractional $g$-Laplacian with $g$ is the antiderivative of a Young function and $f$ in suitable Orlicz space. This includes the mixed fractional $(p,q)-$Laplacian as a special case. The solution so obtained are also shown to be locally H\"older continuous.
	\end{abstract}
	
	%\pagenumbering{roman}
	\maketitle

	%\date{}
	%\dedicatory{}
	%\tableofcontents
	%\pagenumbering{arabic}
	%\setcounter{page}{1}
	%\pagestyle{myheadings}

	\section{Introduction}
	Nonlocal problems have been a subject of immense interest in mathematics recently. Various studies have been published to verify if the results of the Laplace operator can be suitably generalized for problems involving fractional Laplacian and its generalization. Continuing with the spirit of recent developments in the study of nonlocal operators, in this article, we consider the following problem
	\begin{equation}\label{eq problem}
		\begin{split}
			(-\D_g)^su(x)&=f(x)u(x)^{-q(x)}\; \mbox{in } \Om,\\
			u&>0\;  \mbox{in }\; \Om,\\
			u&=0\;  \mbox{in }\; \RR^N\setminus\Om
		\end{split}
	\end{equation} 
	with $\Om$ being a smooth bounded domain in $\RR^N$ and $q$ is a non-negative $C^1$ function in $\overline{\Om}$, and the fractional $g$-Laplacian operator is defined as
	$$
	(-\D_g)^su(x):=\I{\RR^N}g\left(\frac{u(x)-u(y)}{|x-y|^s}\right)\frac{dy}{|x-y|^{N+s}}
	$$ 
	with $g:[0,\infty)\to\RR$ is a right continuous function satisfying the following assumptions:
	\begin{enumerate}[label=($H_{\alph*}$)]
		\item $g(0)=0;\;g(t)>0$  for $t>0$ and $\lim_{t\to+\infty}g(t)=\infty$.
		\item $g$ is convex on $(0,\infty)$.
		\item $g'$ is nondecreasing on $(0,\infty)$, and hence on $\RR\setminus \{0\}$.
	\end{enumerate}
	Given $g:\RR\to\RR$, we define $G:[0,\infty)\to[0,\infty)$, called the N-function or Young's function by 
	$$
	G(t):=\I{0}^t g(\tau)d\tau.
	$$
	We also assume the following additional conditions on $G$ and $g$:
	%\begin{assumption}\hfill
	\begin{enumerate}[label=$(H_\alph*)$,start=5]
		\item $g=G'$ is absolutely continuous, so it is differentiable almost everywhere.
		\item $ \I{0}^1\frac{G^{-1}(\tau)}{\tau^{\frac{N+s}{N}}}d\tau<\infty$ and $\I{1}^\infty\frac{G^{-1}(\tau)}{\tau^{\frac{N+s}{N}}}d\tau=\infty$
		\item There exist $p^+,p^-$ such that 
		$$
		1<p^--1\leq \frac{tg'(t)}{g(t)}\leq p^+-1\leq \infty \quad t>0.
		$$
	\end{enumerate}
	%\end{assumption}
	Note that we will always be assuming conditions $(H_a)-(H_g)$ on $g \;\mbox{and}\; G$ throughout the whole paper until otherwise specified. In literature, $G$ is known as a Young function or an $N$-function.
	
	\begin{remark}
		The following examples of $G$ fits our framework:
		\begin{enumerate}[label=(\roman*)]
			\item  $G_{p}(t):=\frac{1}{p}t^{p}$, where $p\geq 2$. 
			\item If one takes $G_{p_1,p_2}(t):=\frac{1}{p_1}t^{p_1}+\frac{1}{p_2}t^{p_2}$, where $p_1,p_2\geq 2$. One gets, $$(-\D_{g_{p_1,p_2}})^s=(-\D_{p_1})^s+(-\D_{p_2})^s.$$
			\item For $a,b,c>0$ and $g(t)=t^a\log(b+ct)$ we get, $$G(t)=\frac{t^{1+a}}{(1+a)^2}[H^2_1(1+a,1,2+a,-\frac{ct}{b})+(1+a)\log(b+ct)-1]$$
			with $p^-=1+a,\;p^{+}=2+a$, where $H^2_1$ is a hyper geometric function.
		\end{enumerate}
	\end{remark}
	Before we start with the preliminaries, let us briefly recall some related literature concerning the singular problems. Singular problems have a long history starting from the seminal work of Crandall-Rabinowitz-Tartar \cite{CrRaTa}, where for a suitably regular $f$, the problem $-\Delta u= f(x)u^{-\delta}$ was considered in a bounded domain and is shown to admit a classical solution irrespective of the sign of $\delta>0$, subject to Dirichlet boundary condition. The classical solution so obtained was shown to be the weak solution provided $0<\delta<3$ in another celebrated work of Lazer-Mckenna \cite{LaMc}. Singularly perturbed problems were also studied in \cite{Gia1, Gia2} and the reference therein. The case of $f\in L^p(\Omega),\;p\geq 1$ was first treated in Boccardo-Orsina \cite{BoOr}, who showed the existence and regularity results for different cases of $m$ and $\delta$. One may find the p-Laplace generalization of Boccardo-Orsina's work in Scuinzi et al. \cite{CaScTr}, where the delicate issue of uniqueness was also addressed. Anisotropic Laplacians with singular nonlinearities have also been dealt with in several papers, see \cite{Ba1, Ba2, BaGa} to name a few. In \cite{BaBeDe}, the fractional problem given by 
	\begin{align}
		(-\D)^su(x)&=\lambda f(x)u(x)^{-\gamma} +Mu^p\;\mbox{in } \Om,\nonumber\\
		u&>0 \;\mbox{in } \Om,\label{eq problem-0}\\
		u&=0 \;\mbox{in } \RR^N\setminus\Om,\nonumber
	\end{align}
	was first considered under the condition that $n
	>2s,M\geq 0, 0<s<1, 1<p<2_s^{*}-1$ and shown to admit a distributional solution for $f\in L^m(\Om)$ and $\lambda>0$ small.
	In \cite{CaMoScSq}, the authors studied the problem
	\begin{align*}
		(-\D_p)^su(x)&=f(x)u(x)^{-\gamma} \;\mbox{in } \Om,\nonumber\\
		u&>0 \;\mbox{in } \Om,\label{eq problem-2}\\
		u&=0 \;\mbox{in } \RR^N\setminus\Om, \nonumber
	\end{align*}
	and proved the existence and uniqueness results. The first instance, to the best of our knowledge of studying variable exponent singularities was in Carmona-Mart\'{\i}nez-Aparicio \cite{CaMa} where the rather surprising phenomenon of the restriction of the nonlinearity to $1$ near the boundary of the domain for a weak solution was studied in contrast to the constant exponent where such restriction is imposed on the whole domain. Similar problems involving fractional p-Laplacian with variable exponent may be found in Garain-Tuhina \cite{GaMu}. However, by considering the Orlicz setup, we can address a large class of interesting problems in one go.
	In the next section, we give some preliminary results which we are already known. 
	
	\section{Preliminaries}
	
	Let us start by introducing the reader to the functional setup related to the fractional Orlicz-Sobolev spaces. A detailed discussion can be found in \cite{BMRS, BoSa, BaOuTa}. Throughout the section, we shall assume $\Om$ to be a bounded domain and $s\in(0,1)$. Throughout the rest of the article, $C$ will stand for a generic constant, which may vary in each of its appearances. First, we define the modular functions:
	$$ M_{L^G(\Om)}(f):= \I{\Om} G(|f(x)|) dx \  \textrm{ and } M_{W^{s,G}(\Om)}(f):= \I{\Om}\I{\Om}G\left( \frac{| f(x)-f(y) |}{| x-y |^s}\right) \frac{dxdy}{|x-y|^N}. $$
	The Banach space 
	$$
	L^G(\Om):=\left\{ f: \Om \to \RR \  \textrm{measurable} \ \Big| \ \exists \ \lambda>0 \mbox{ such that } M_{L^G(\Om)}\left(\frac{f}{\lambda}\right)<\infty\right\} 
	$$ is called the Orlicz space. This space is equipped with the norm
	$$
	\|f\|_{L^G(\Om)}:=\inf\{\lambda>0\ |\ M_{L^G(\Om)}\left(\frac{f}{\lambda}\right)\leq 1\}.
	$$
	The infimum in the above definition is known to be achieved.
	The fractional Orlicz-Sobolev spaces are defined as
	$$
	W^{s,G}(\Om):=\left\{ f\in L^G(\Om) \ \Big| \ \exists \ \lambda>0 \mbox{ such that } M_{W^{s,G}(\Om)}\left(\frac{f}{\lambda}\right)<\infty\right\}.
	$$
	This space is equipped with the seminorm
	$$
	\|f\|_{W^{s,G}(\Om)}:=\inf\{\lambda>0\ |\ M_{W^{s,G}(\Om)}\left(\frac{f}{\lambda}\right)\leq 1\}.
	$$
	However, we shall mainly be working with the spaces defined by
	$$\hat{W}^{s,G}(\Om):=\left\{ f\in L^{G}_{loc}(\RR^N):\; \exists\;U\Subset\Om\;\mbox{s.t}\; ||u||_{s,G,U}+\int_{\RR^N}g(\frac{|u(x)|}{1+|x|^s})\;\frac{dx}{(1+|x|)^{n+s}}<\infty\right\}$$
	and,
	$$
	W^{s,G}_0(\Om):=\left\{ f\in W^{s,G}(\RR^N) \ \Big| \  f\equiv 0\ \mbox{ on }\RR^N\setminus \Om\right\}.
	$$
	$W^{s,G}_0(\Om)$ is equipped with the norm $\|\cdot\|_{W^{s,G}(\RR^N)}$.
	Note that for $G(t)=t^p;\;1<p<\infty$, $L^G(\Om)$ and $W^{s,G}(\Om)$ are well known Lebesgue space $L^p(\Om)$ and the fractional Sobolev space $W^{s,p}(\Om)$ respectively (see \cite[p.~524]{hhg}).
	 
	We now discuss some properties of these spaces which we shall use in the next section. We start by observing that the assumption $(H_g)$ implies
	\begin{equation}\label{eq growth}	
		2<p^-\leq \frac{tg(t)}{G(t)}\leq p^+<\infty, \quad t>0.
	\end{equation}
	To see this, note that assumption $(H_g)$ implies $(tg(t))'\leq p^+G(t)'$.
	The following two lemmas will be used frequently in the rest of the article.
	
	\begin{lemma}\label{lemma delta2}
		Let $G$ be an $N$-function, and let $g=G'$ satisfy $(H_a)-(H_g)$. Then
		$$
		\lambda^{p^-}G(t) \leq G(\lambda t) \leq \lambda^{p^+}G(t)\quad \forall\ \lambda\geq1, \ \forall t>0,
		$$
		where $p^+,p^-$ is the constant as defined in $(H_g)$.
		The above inequality is equivalent to
		$$
		\lambda^{p^-}G(t) \geq G(\lambda t) \geq \lambda^{p^+}G(t)\quad \forall\ 0 \leq \lambda\leq1, \ \forall t>0.
		$$
	\end{lemma}
	\begin{proof}
		For any $\lambda>1$,
		$$
		\log(\lambda^{p^-})
		=\I{t}^{\lambda t}\frac{p^-}{\tau}d\tau
		\leq \I{t}^{\lambda t}\frac{g(\tau)}{G(\tau)}d\tau
		\leq \I{t}^{\lambda t}\frac{p^+}{\tau}d\tau
		=\log(\lambda^{p^+}).
		$$
		This implies
		$$
		\log \left(\lambda^{p^-}\right) \leq \log \left(\frac{G(\lambda t)}{G(t)}\right) \leq \log \left(\lambda^{p^+}\right).
		$$
		The lemma follows.
	\end{proof}
	
	An immediate consequence of \cref{lemma delta2} is the following
	
	\begin{lemma}\label{lemma equivalent norm}
		When $\|f\|_{W^{s,G}( \Om )}\leq1$, $$\|f\|_{W^{s,G}( \Om )}^{p^+} \leq M_{W^{s,G}( \Om )}(f) \leq \|f\|_{W^{s,G}( \Om )}^{p^-},$$ and when $\|f\|_{W^{s,G}( \Om )}\geq1$, $$\|f\|_{W^{s,G}( \Om )}^{p^-} \leq M_{W^{s,G}( \Om )}(f) \leq \|f\|_{W^{s,G}( \Om )}^{p^+}.$$
	\end{lemma}
	
	\begin{lemma}\label{lemma lindqvist}
		Let $G$ be an N-function satisfying  $(H_a)-(H_g)$. For any two real numbers $a$ and $b$, we have
		$$
		(g(b)-g(a))(b-a)
		\geq C(G)\;G(|b-a|).
		$$
		for some constant $C$ depending on the $N-$function $G$.
	\end{lemma}
	\begin{proof}
		By the symmetry of the inequality, it is enough to prove this lemma for the cases $0<a\leq b$ and $a<0<b$. In the first case, using Taylor's theorem with an integral form of reminder, we have
		\begin{align*}
			G(|b-a|)
			&=G(0)+g(0)|b-a|+\frac{1}{2}\I{0}^{b-a}g'(t)(b-a-t)dt\\
			&= \frac{b-a}{2}\I{a}^{b}g'(t-a)\frac{b-t}{b-a}dt
			\leq \frac{b-a}{2}\I{a}^{b}g'(t)dt\\
			&=\frac{(b-a)(g(b)-g(a))}{2}
		\end{align*}
		the case $0 \geq a \geq b$ follows similarly.

		Now suppose $a<0<b$. Using convexity of $G$, we get
		\begin{align*}
			G(\frac{|b-a|}{2})
			&= G(\frac{b+(-a)}{2})
			\leq \frac{1}{2}(G(b)+G(-a))
			\leq \frac{1}{2}(\frac{bg(b)}{p^-}+\frac{(-a)g(-a)}{p^-})\\
			&\leq \frac{1}{2p^-} (bg(b)+ag(a)-ag(b)-bg(a))
			= \frac{1}{2p^-} (g(b)-g(a))(b-a).
		\end{align*}
	\end{proof}

	\begin{definition}\label{def conjugate}
		Let $G$ be an $N$-function.
		\begin{enumerate}
			\item The $N$-function $\overline{G}$ is called the conjugate of $G$, and is defined by
			$$
			\overline{G}(t):=\I{0}^t\overline{g}(\tau)d\tau,
			$$
			where $\overline{g}(t):=\sup \{\tau\ \Big|\ g(\tau)\leq t \}$.
			
			\item The $N$-function $G_*$, defined by
			$$
			G_*^{-1}(t):=\I{0}^t\frac{G^{-1}(\tau)}{\tau^{\frac{N+s}{N}}}d\tau,
			$$
			is called the Sobolev conjugate of $G$.
			
			\item An $N$-function $G$ is said to be essentially stronger than $H$, written as $H\prec\prec G$, if for any $k>0$,
			$$
			\lim_{t\to\infty}\frac{H(kt)}{G(t)}=0.
			$$ 
		\end{enumerate}  
	\end{definition}
	\begin{lemma}[H\"older Inequality]\label{holder}
		Let $G$ be an $N$-function, $N\geq1$, and $\Om\subseteq \RR^N$. Then, we have for any $u,v:\Om\to\RR$,
		$$
		\I{\Om}|uv|
		\leq
		\|u\|_{L^G(\Om)} \|v\|_{L^{\overline{G}}(\Om)}
		$$
	\end{lemma}
	\begin{lemma}[{\cite[Corollary~6.2]{BoSa}}]\label{poincare}
		Let $G$ be an $N$-function which satisfy the $\D_2$-condition. Then there exists a constant $C=C(n,G,\Om)$ such that for any $u\in W_0^{s,G}(\Om)$,
		$$
		\I{\Om}G(u(x))dx
		\leq C\I{\RR^N}\I{\RR^N}G\left(\frac{u(x)-u(y)}{|x-y|^s}\right)\frac{dxdy}{|x-y|^N}.
		$$
	\end{lemma}
	\begin{lemma}[{\cite[Theorem~1]{BaOu} and \cref{poincare}}]\label{lemma embedding}
		Let $G$ be an $N$-function, and let $\Om$ be a bounded open subset of $\RR^N$ with $C^{0,1}$-regularity. Then we have the following:
		\begin{enumerate}
			\item the embedding
			$W^{s,G}_0(\Om) \to L^{G_*}(\Om)$, is continuous.
			\item Moreover, for any $N$-function $H$, the embedding
			$
			W^{s,G}_0(\Om) \to L^H(\Om)
			$
			is compact if $H\prec\prec G$.
		\end{enumerate}
	\end{lemma}

	%\begin{lemma}[Strong maximum principle, see \cite{MoSa}]\label{lemma smp}
	%	Let $\Om\subset\RR^N$ be a bounded domain. Assume that $u \in W^{s,A}_0(\Om)$ and satisfies
	%	\begin{equation*}
		%		\begin{cases}
			%			(-\D_g)^s u(x) \geq 0 \quad \mbox{if}\; x \in \Om\\
			%			u(x) \geq 0 \quad 
			%			\mbox{if } x \in \RR^N \setminus \Om.
			%		\end{cases}
		%	\end{equation*}
	%	Then $u \geq 0$ in $\Om$. Moreover, if $u(x) = 0$ at some point $x %\in \Om$, then $u\equiv0$ in $\RR^N$.
	\begin{lemma}[Weak Harnack Inequality, see \cite{Bo}]\label{lemma smp}
		If $u\in \hat{W}^{s,G}(B_{3^{-1}R})$ satisfies weakly 
		\begin{equation*}
			\begin{cases}
				(-\D_g)^s u(x) &\geq 0 \quad \mbox{if}\; x \in \Om\\
				\hfill u(x) &\geq 0 \quad 
				\mbox{if } x \in \RR^N
			\end{cases}
		\end{equation*}
		then there exists $\sigma\in (0,1)$ such that $$\inf_{B_{4^{-1}R}} u \geq \sigma R^sg^{-1} (\avint_{B_R\setminus B_{2^{-1}R}}g(R^{-s}|u|)\;dx)$$
	\end{lemma}
	
	\section{Main Results}
	
	We begin this section by stating the definition of our weak solution
	\begin{definition}[Weak solutions]
		The function $u\in \hat{W}^{s,G}(\Om)$ is said to be a weak solution of \cref{eq problem} if $u>0$ in $\Om$, and for any $\varphi\in\test(\Om)$ one has, $\frac{f}{u^{q(.)}}\in L^1_{loc}(\Om)$ and
		\begin{equation}\label{eq weak solution}
			\I{\RR^N}\I{\RR^N}g\left(\frac{u(x)-u(y)}{|x-y|^s}\right)\frac{(\varphi(x)-\varphi(y))}{|x-y|^{N+s}}\;dxdy=\I{\Om}f(x)u(x)^{-q(x)}\phi(x)\;dx.
		\end{equation}
		The boundary condition is understood in the sense that 
		\begin{enumerate}
			\item  if $q(x)\leq1$ on $\Om_\delta:=\{ x\in\Om\ \Big| \ \mbox{dist}(x,\del \Om)<\delta \}$, then $u\in W^{s,G}_0(\Om).$
			\item Elsewhere one has, $\Phi(u)\in W^{s,G}_0(\Om)$, where 
			$$\Phi(t):=\I{0}^tG^{-1}\left(G(1)\tau^{q^*-1}\right)d\tau.$$
		\end{enumerate}
		
		Furthermore, we say that $u$ is a subsolution (or supersolution) of \cref{eq problem} if, for any $\varphi\in\test(\Om)$,
		\begin{equation}\label{eq subsolution}
			\I{\RR^N}\I{\RR^N}g\left(\frac{u(x)-u(y)}{|x-y|^s}\right)\frac{(\varphi(x)-\varphi(y))}{|x-y|^{N+s}}\;dxdy\leq \ (\mbox{or } \geq)\ \I{\Om}f(x)u(x)^{-q(x)}\varphi(x)\;dx.
		\end{equation}
	\end{definition}
	\begin{theorem}\label{main 1}
		Let there exist $\delta>0$ such that $q(x)\leq1$ on $\Om_\delta:=\{ x\in\Om\ \Big| \ \mbox{dist}(x,\del \Om)<\delta \}$ and $f\in L^{\overline{G_*}}(\Om)$. Then \cref{eq problem} has a weak solution in $W^{s,G}_0(\Om)$ with $\mbox{essinf}_K u>0$ for any $K\Subset\Om$.
	\end{theorem}
	\begin{theorem}\label{main 2}
		Let $g$ is sub-multiplicative and there exist $q^*>1,\ \delta>0$ such that $\|q\|_{L^\infty(\Om_\delta)}\leq q^*$ and let $$H(t):=G_*(t^\frac{{p^-}+q^*-1}{p^-q^*})$$ be an $N$-function such that $f\in L^{\overline{H}}(\Om)$. Then \cref{eq problem} has a weak solution $u\in W^{s,G}_{loc}(\Om)$ with $\mbox{essinf}_K u>0$ for any $K\Subset\Om$ such that $\Phi(u)\in W^{s,G}_0(\Om)$, where 
		$$\Phi(t):=\I{0}^tG^{-1}\left(G(1)\tau^{q^*-1}\right)d\tau.$$
	\end{theorem}
	
	\begin{theorem}\label{main 3}
		Every weak solution of \cref{eq problem}  obtained through \cref{main 1,main 2} belongs to $C_{loc}^\alpha(\Om)$ for some $\alpha \in (0, 1)$.
	\end{theorem}
	
	We now develop some results which are needed to prove \cref{main 1} and \cref{main 2}.
	\begin{lemma}[Comparison Principle]\label{comparison}
		Let $u,v\in C(\RR^N)$ with $[u]_{W^{s,G}(\RR^N)}, [v]_{W^{s,G}(\RR^N)}<\infty$, and $D\subseteq \RR^N$ be a domain such that $|\RR^N\setminus D|> 0$. If $v \geq u$ in $\RR^N \setminus D$, and 
		$$
		\I{\RR^N}\I{\RR^N}g\left(\frac{v(x)-v(y)}{|x-y|^s}\right)\frac{\varphi(x)-\varphi(y)}{|x-y|^{N+s}}dxdy
		\geq \I{\RR^N}\I{\RR^N}g\left(\frac{u(x)-u(y)}{|x-y|^s}\right)\frac{\varphi(x)-\varphi(y)}{|x-y|^{N+s}}dxdy
		$$
		for $\varphi=(u-v)^+$, then $v \geq u$ in $\RR^N$.
	\end{lemma}
	
	\begin{proof}
		We need to show that $v \geq u$ in $D$. The two integrals can be shown to be finite using H\"older's inequality and the assumptions on $g$. Then we have
		$$
		\I{\RR^N}\I{\RR^N}\left[ g\left(\frac{v(x)-v(y)}{|x-y|^s}\right)-g\left(\frac{u(x)-u(y)}{|x-y|^s}\right)\right]\frac{\varphi(x)-\varphi(y)}{|x-y|^{N+s}}dxdy
		\geq 0.
		$$
		This, and the identity
		\begin{equation*}
			g(t_2)-g(t_1)=(t_2-t_1)\I{0}^{1}g'((t_2-t_1)\tau+t_1)d\tau,
		\end{equation*}
		gives
		\begin{equation}\label{eq-cp-1}
			\I{\RR^N}\I{\RR^N}\left(v(x)-v(y)-u(x)+u(y)\right)Q(x,y) \frac{\varphi(x)-\varphi(y)}{|x-y|^{N+2s}}dxdy
			\geq 0,
		\end{equation}
		where
		$$
		Q(x,y):=\I{0}^{1}g'\left(\frac{(v(x)-v(y)-u(x)+u(y))\tau+u(x)-u(y)}{|x-y|^s}\right)d\tau.
		$$
		From the assumption on $g$, we know $g'\geq 0$. So $Q(x,y)\geq0$, and $Q(x,y)=0$ if and only if the integrand is identically zero. Again this happens if and only if $v(x)=v(y)$ and $u(x)=u(y)$.
		
		Choose $\varphi=(u-v)^+$ and $\psi:=u-v$. \Cref{eq-cp-1}, then, becomes
		\begin{equation}\label{eq-cp-2}
			\I{\RR^N}\I{\RR^N}Q(x,y) \frac{(\varphi(x)-\varphi(y))(\psi(y)-\psi(x))}{|x-y|^{N+2s}}dxdy
			\geq 0.
		\end{equation}
		We can see that, after choosing $\varphi:=(u-v)^+$, and using the fact that $\psi^+(y)\psi^-(y)=0$,
		$$
		(\varphi(x)-\varphi(y))(\psi(y)-\psi(x))=-(\psi^+(x)-\psi^+(y))^2-\psi^-(y)\psi^+(x)-\psi^-(x)\psi^+(y)\leq 0.
		$$
		This along with \cref{eq-cp-2}, and the fact that $Q(x,y)\geq0$ implies $Q(x,y)=0$ or $-(\psi^+(x)-\psi^+(y))^2-\psi^-(y)\psi^+(x)-\psi^-(x)\psi^+(y)=0$ almost everywhere. In both the cases, we must have $\psi^+(x)=\psi^+(y)$ for a.e. $(x,y)$. Since $(u-v)^+=0$ on $\RR^N\setminus D$, by continuity of $u,v$, we conclude that $\psi^+=0$ on $\RR^N$. This implies $v\geq u$ on $\RR^N$.
	\end{proof}
	
	\begin{lemma}\label{lemma convex}
		Let $g$ be sub-multiplicative, that is, there is a constant $C>0$ for which $C g(t_1t_2)\leq g(t_1)g(t_2)$ for any $t_1,t_2>0$. Let $F$ and $u$ be such that 
		$$
		\I{\RR^N}\I{\RR^N}g\left(\frac{u(x)-u(y)}{|x-y|^s}\right)\frac{\varphi(x)-\varphi(y)}{|x-y|^{N+s}}\;dxdy=\I{\Om}F\varphi\; dx,
		$$
		for any $\varphi\in W^{s,G}_0(\Om)$. Then for any convex and Lipschitz function $\Phi$, we have
		\begin{equation*}
			\I{\Om}F(x)g(\Phi'(u(x)))\Phi(u)\; dx 
			\geq C\I{\RR^N}\I{\RR^N}G\left(\frac{|\Phi(u(x))-\Phi(u(y))|}{|x-y|^s}\right)\frac{dxdy}{|x-y|^N}.
		\end{equation*}
	\end{lemma}
	\begin{proof}
		First, note that, by density argument, we can assume $\Phi$ to be $C^1$. Choose $\varphi=g(\Phi'(u))\psi$. Then we have 
		
		\begin{align*}
			2\II{\{u(x)>u(y)\}}&g\left(\frac{u(x)-u(y)}{|x-y|^s}\right)\frac{g(\Phi'(u(x)))\psi(x)-g(\Phi'(u(y)))\psi(y)}{|x-y|^{N+s}}dxdy\\
			&=\I{\RR^N}\I{\RR^N}g\left(\frac{u(x)-u(y)}{|x-y|^s}\right)\frac{g(\Phi'(u(x)))\psi(x)-g(\Phi'(u(y)))\psi(y)}{|x-y|^{N+s}}dxdy\\
			&=\I{\Om}F(x)g(\Phi'(u(x)))\psi(x) dx.
		\end{align*}
		Set $u(x)=a$, $u(y)=b$, $\psi(x)=A$ and $\psi(y)=B$. Then the integrand in the LHS becomes
		$$
		g\left(\frac{a-b}{|x-y|^s}\right)\frac{g(\Phi'(a))A-g(\Phi'(b))B}{|x-y|^{N+s}}.
		$$
		Using the convexity of $\Phi$, we have 
		$$
		\Phi(a)-\Phi(b)\leq \Phi'(a)(a-b) \mbox{ and } \Phi(a)-\Phi(b)\geq \Phi'(b)(a-b).
		$$
		We then have 
		\begin{align*}
			&g\left(\frac{a-b}{|x-y|^s}\right)\frac{g(\Phi'(a))A-g(\Phi'(b))B}{|x-y|^{N+s}}\\
			&\geq g\left(\frac{a-b}{|x-y|^s}\right)\frac{g\left(\frac{\Phi(a)-\Phi(b)}{a-b}\right)A-g\left(\frac{\Phi(a)-\Phi(b)}{a-b}\right)B}{|x-y|^{N+s}}\\
			&= g\left(\frac{a-b}{|x-y|^s}\right)g\left(\frac{\Phi(a)-\Phi(b)}{a-b}\right)\frac{A-B}{|x-y|^{N+s}}\\
			&\geq Cg\left(\frac{\Phi(a)-\Phi(b)}{|x-y|^s}\right)\frac{A-B}{|x-y|^{N+s}}.
		\end{align*}
		This, after taking $\psi=\Phi(u)$ (note that $\Phi$ is assumed to be $C^1$), gives
		
		\begin{equation*}
			\I{\Om}F(x)g(\Phi'(u(x)))\Phi(u) dx 
			\geq C\I{\RR^N}\I{\RR^N}G\left(\frac{|\Phi(u(x))-\Phi(u(y))|}{|x-y|^s}\right)\frac{dxdy}{|x-y|^N}.
		\end{equation*}
	\end{proof}

	\begin{lemma}\label{laplace equation}
		Let $f\in L^\infty(\Om)$ with $f\geq 0$, and $f$ is not identically zero. Then the problem
		\begin{equation}\label{eq sq problem}
			\begin{cases}
				(-\D_g)^su=f, \quad \mbox{ in } \Om,\\
				u>0, \mbox{ in } \Om,\\
				u=0, \mbox{ in } \RR^N\setminus\Om
			\end{cases}
		\end{equation}
		has a unique solution $u\in W^{s,G}_0(\Om)\cap L^\infty (\Om)$.	
	\end{lemma}
	\begin{proof}
		The existence, uniqueness, and continuity follows from \cite[Theorem~6.16]{BoSa},  \cref{lemma smp}, and the fact that $f\geq0$, so that $(-\D_g)^su\geq0$ on $\Om$, using \cref{comparison}. It remains to show that $u\in L^\infty(\Om)$. For this, we shall assume, without loss of generality, that $\Om\subseteq B(0,1)$ and fix $\alpha>1$.
		
		Let us consider 
		$$
		v_\alpha(x)=
		\begin{cases}
			\alpha(1-|x|),\quad &\mbox{when } |x|<1,\\
			0, \quad &\mbox{otherwise.} 
		\end{cases}
		$$
		Note that for since $\alpha>1$, for any $0<\lambda<1$ we have, using \cref{lemma delta2,eq growth}, the estimate $g(\alpha\lambda t)>\frac{p^-\alpha^{p^--1}\lambda^{p^+-1}G(t)}{ t}$ when $t>0$. Again, for $x\in \Om\subseteq B(0,1)\subseteq B(x,1+|x|)$ we get
		\begin{align*}
			(-\D_g)^s v_\alpha(x)&\geq \I{|y|>1}g\left(\frac{v_\alpha(x)-v_\alpha(y)}{|x-y|^s}\right)\frac{dy}{|x-y|^{N+s}}\\
			&=\I{|y|>1}g\left(\frac{v_\alpha(x)}{|x-y|^s}\right)\frac{dy}{|x-y|^{N+s}}\\
			&\geq p^-\alpha^{p^--1}(1-|x|)^{p^+-1} \I{|y|>1}G\left(\frac{1}{|x-y|^s}\right)\frac{dy}{|x-y|^N}\\
			&= p^-\alpha^{p^--1}(1-|x|)^{p^+-1} \I{|y|>1}G\left(\frac{1}{(1+|y|)^s}\right)\frac{dy}{(1+|y|)^N}
			\to \infty
		\end{align*}
		uniformly as $\alpha \to \infty$. Thus, as $f$ is bounded, we can choose $\alpha$ large enough to get $(-\D)^s_g v_\alpha>(-\D)^s_g u$. Applying \cref{comparison} we get $u\leq v_\alpha$ in $\RR^N$. Thus, $u$ is bounded. 
	\end{proof}

	We consider the following approximated problem of \cref{eq problem}, where we used the notation, $f_n=\min\{f,n\}$ for all $n\in\NN$, and assumed $q>0$ is $C^1$,
	\begin{align}\label{eq approx problem}
		(-\D_g)^su(x)&=\frac{f_n(x)}{(u(x)+\frac{1}{n})^{q(x)}}\;\mbox{in }\;\Om,\\
		\nonumber u&>0\;\mbox{in }\;\Om,\\
		\nonumber u&=0\;\mbox{in }\;\RR^N\setminus\Om.
	\end{align}
	
	\begin{lemma}\label{lemma approx problem}
		For a fixed $n\in\NN$, \cref{eq approx problem} has a weak solution $u_n\in C^{\alpha (n)}(\Om)$ where $\alpha(n)\in (0,1)\;\forall n\in\NN$.
	\end{lemma}
	\begin{proof}
		Note that $\frac{f_n(x)}{(u^+(x)+\frac{1}{n})^{q(x)}}\in L^\infty(\Om)$. Hence by \cref{laplace equation}, there exists a unique solution $w\in W^{s,G}_0(\Om)\cap L^\infty(\Om)$ to the problem 
		\begin{equation}\label{eq approx problem positive}
			\begin{split}
				(-\D_g)^sw(x)&=\frac{f_n(x)}{(u^+(x)+\frac{1}{n})^{q(x)}}\;\mbox{in }\; \Om,\\
				\nonumber w&>0\;\mbox{in }\; \Om,\\
				\nonumber w&=0\;\mbox{in }\; \RR^N\setminus\Om.
			\end{split}
		\end{equation}
		This allows us to define the operator $S:W^{s,G}_0(\Om)\to W^{s,G}_0(\Om)$ by $S(u)=w$ the solution of \cref{eq approx problem positive}. Multiplying both sides of \cref{eq approx problem positive} by $w$, we get
		\begin{multline*}
			\I{\RR^N}\I{\RR^N}g\left(\frac{w(x)-w(y)}{|x-y|^s}\right)\frac{(w(x)-w(y))}{|x-y|^{N+s}}dxdy
			=\I{\Om}\frac{f_n(x)w(x)}{(u(x)^++\frac{1}{n})^{q(x)}}dx
			\leq n^{1+\|q\|_{L^\infty(\Om)}}\|w\|_{L^1(\Om)}.
		\end{multline*}
		Applying \cref{eq growth}, we get 
		\begin{align*}
			\I{\RR^N}\I{\RR^N}G\left(\frac{|w(x)-w(y)|}{|x-y|^s}\right)\frac{dxdy}{|x-y|^{N}}
			&\leq \frac{1}{p^-}\I{\RR^N}\I{\RR^N}g\left(\frac{|w(x)-w(y)|}{|x-y|^s}\right)\frac{|w(x)-w(y)|}{|x-y|^{N+s}}dxdy\\
			&= \frac{1}{p^-}\I{\RR^N}\I{\RR^N}g\left(\frac{(w(x)-w(y))}{|x-y|^s}\right)\frac{(w(x)-w(y))}{|x-y|^{N+s}}dxdy\\
			&\leq \frac{n^{1+\|q\|_{L^\infty(\Om)}}}{p^-}\|w\|_{L^1(\Om)}.
		\end{align*}
		Assume $\|w\|_{W_0^{s,G}(\Om)}> 1,$
		\begin{multline*}
			\frac{1}{\|w\|_{W_0^{s,G}(\Om)}^{p^-}}\I{\RR^N}\I{\RR^N}G\left(\frac{|w(x)-w(y)|}{|x-y|^s}\right)\frac{dxdy}{|x-y|^{N}}
			\geq \I{\RR^N}\I{\RR^N}G\left(\frac{|w(x)-w(y)|}{\|w\|_{W_0^{s,G}(\Om)}|x-y|^s}\right)\frac{dxdy}{|x-y|^{N}}
			=1.
		\end{multline*}
		So, we have
		$$
		\|w\|_{W_0^{s,G}(\Om)}^{p^-}\leq \frac{n^{1+\|q\|_{L^\infty(\Om)}}}{p^-}\|w\|_{L^1(\Om)},
		$$
		and consequently, by \cref{lemma embedding},
		$$\|w\|_{W_0^{s,G}(\Om)}^{p^--1}\leq Cn^{1+\|q\|_{L^\infty(\Om)}}$$
		provided $\|w\|_{W_0^{s,G}(\Om)}>1$. Setting $R:=\max\{1,\left(Cn^{1+\|q\|_{L^\infty(\Om)}}\right)^\frac{1}{p^--1}\}$, we can see that $S$ maps the ball of radius $R$ in the metric space $W_0^{s,G}(\Om)$, into itself.
		The proof will now be complete if we show that $S$ is continuous and compact.\\
		
		\textbf{Proof of continuity of $S$:}
		Assume that $u_i\to u$ in $W_0^{s,G}(\Om)$. Set $w_i=S(u_i)$ and $w=S(u)$. So that we have for any $\varphi\in W^{s,G}_0(\Om)$,
		\begin{align}
			\I{\RR^N}\I{\RR^N}g\left(\frac{w_i(x)-w_i(y)}{|x-y|^s}\right)\frac{(\varphi(x)-\varphi(y))}{|x-y|^{N+s}}dxdy
			&=\I{\Om}\frac{f_n(x)\varphi(x)}{(u_i(x)^++\frac{1}{n})^{q(x)}}dx \quad \mbox{and} \label{eq1}
			\\
			\I{\RR^N}\I{\RR^N}g\left(\frac{w(x)-w(y)}{|x-y|^s}\right)\frac{(\varphi(x)-\varphi(y))}{|x-y|^{N+s}}dxdy
			&=\I{\Om}\frac{f_n(x)\varphi(x)}{(u(x)^++\frac{1}{n})^{q(x)}}dx\label{eq2}.
		\end{align}
		
		We have to show that $w_i\to w$ in $W_0^{s,G}(\Om)$. By \cref{lemma embedding}, passing to a subsequence, $u_i\to u$ in $L^{G_*}(\Om)$ and $u_i\to u$ a.e. in $\Om$.
		Set $\overline{w_i}:=w_i-w$. Subtracting \cref{eq2} from \cref{eq1}, with the choice $\varphi=\overline{w_i}$, and then applying \cref{lemma lindqvist} for $a=\frac{w(x)-w(y)}{|x-y|^s}$ and, $b=\frac{w_i(x)-w_i(y)}{|x-y|^s}$, we get
		\begin{multline*}
			C(G)\I{\RR^N}\I{\RR^N}G\left(\frac{|\overline{w_i}(x)-\overline{w_i}(y)|}{|x-y|^s}\right)\frac{dxdy}{|x-y|^N}dxdy\\
			\leq \I{\Om}f_n(x)\left(\frac{1}{(u_i(x)^++\frac{1}{n})^{q(x)}}
			- \frac{1}{(u(x)^++\frac{1}{n})^{q(x)}}\right)(w_i(x)-w(x))dx.
		\end{multline*}
		We apply \cref{lemma equivalent norm} on the left-hand side and H\"older inequality on the right-hand side of this equation to get,
		\begin{align*}
			C(G)&\min\left\{\|w_i-w\|_{W^{s,G}}^{p^+}, \|w_i-w\|_{W^{s,G}}^{p^-}\right\}\\
			&\leq C\left\| f_n(x)\left(\frac{1}{(u_i(x)^++\frac{1}{n})^{q(x)}}
			- \frac{1}{(u(x)^++\frac{1}{n})^{q(x)}}\right) \right\|_{L^{G_*'}}\|w_i-w\|_{L^{G_*}}\\
			&\leq C\left\| f_n(x)\left(\frac{1}{(u_i(x)^++\frac{1}{n})^{q(x)}}
			- \frac{1}{(u(x)^++\frac{1}{n})^{q(x)}}\right) \right\|_{L^{G_*'}}\|w_i-w\|_{W^{s,G}},
		\end{align*}
		where the last inequality follows from \cref{poincare}. This gives
		$$\min\left\{\|w_i-w\|_{W^{s,G}}^{p^+-1}, \|w_i-w\|_{W^{s,G}}^{p^--1}\right\}\leq C\left\| f_n(x)\left(\frac{1}{(u_i(x)^++\frac{1}{n})^{q(x)}}-\frac{1}{(u(x)^++\frac{1}{n})^{q(x)}}\right)\right\|_{L^{G_*'}}.$$
		Now observe that
		$$
		\left|f_n(x)\left(\frac{1}{(u_i(x)^++\frac{1}{n})^{q(x)}}
		- \frac{1}{(u(x)^++\frac{1}{n})^{q(x)}}\right)\right|
		\leq 2n^{q(x)+1}
		\leq 2n^{\|q\|_{L^\infty}+1}.
		$$
		Hence, as $u_i\to u$ pointwise a.e., by DCT it follows that $w_i\to w$ in $W^{s,G}_0$. Thus $S$ is continuous.\\
		
		\textbf{Proof of compactness of $S$:} Assume that $u_i$ is a bounded sequence in $W_0^{s,G}(\Om)$. As before, denote $w_i:=S(u_i)$. We wish to show that $w_i$ has a convergent subsequence in $W_0^{s,G}(\Om)$. From \cref{eq1,holder,lemma equivalent norm}, we get
		\begin{align*}
			\min\left\{\|w_i\|_{W^{s,G}}^{p^+}, \|w_i\|_{W^{s,G}}^{p^-}\right\}
			&\leq
			C(G)\I{\RR^N}\I{\RR^N}G\left(\frac{w_i(x)-w_i(y)}{|x-y|^s}\right)\frac{dxdy}{|x-y|^{N}}\\
			&\leq C(G)\I{\RR^N}\I{\RR^N}g\left(\frac{w_i(x)-w_i(y)}{|x-y|^s}\right)\frac{(w_i(x)-w_i(y))}{|x-y|^{N+s}}dxdy\\
			&=C(G)\I{\Om}\frac{f_n(x)w_i(x)}{(u_i(x)^++\frac{1}{n})^{q(x)}}dx
			\leq n^{1+\|q\|_{L^\infty(\Om)}}\|w_i\|_{L^1(\Om)}\\
			&\leq C n^{1+\|q\|_{L^\infty(\Om)}}\|w_i\|_{W^{s,G}(\Om)}.
		\end{align*}  
		This shows that $w_i$ is a bounded sequence in $W_0^{s, G}(\Om)$. From the boundedness of the two sequences, $u_i,w_i$, we conclude that there exists $u,w\in W_0^{s,G}(\Om)$ such that $u_i\rightharpoonup u$ and $w_i\rightharpoonup w$ in $W_0^{s,G}(\Om)$.
		We now want to show $S(u)=w$, that is for any $\varphi\in\test(\Om)$,
		\begin{equation}\label{eq4}
			\I{\RR^N}\I{\RR^N}g\left(\frac{w(x)-w(y)}{|x-y|^s}\right)\frac{(\varphi(x)-\varphi(y))}{|x-y|^{N+s}}dxdy\\
			=\I{\Om}\frac{f_n(x)\varphi(x)}{(u(x)^++\frac{1}{n})^{q(x)}}dx.
		\end{equation}
		Note that we already know
		\begin{equation}\label{eq3}
			\I{\RR^N}\I{\RR^N}g\left(\frac{w_i(x)-w_i(y)}{|x-y|^s}\right)\frac{(\varphi(x)-\varphi(y))}{|x-y|^{N+s}}dxdy\\
			=\I{\Om}\frac{f_n(x)\varphi(x)}{(u_i(x)^++\frac{1}{n})^{q(x)}}dx.
		\end{equation}
		By DCT, it is seen easily that the right-hand side of \cref{eq3} converges to the right-hand side of \cref{eq4}. It remains to show the convergence of the left-hand side. Note that,
		$$ 
		\overline{G}(g(t))
		= \I{0}^{g(t)}g^{-1}(\tau)d\tau
		=\I{0}^t\tau g'(\tau)d\tau
		\equiv \I{0}^tg(\tau)d\tau
		=G(t).
		$$
		Using this and the fact that $w_i$'s are bounded in $W^{s,G}_0(\Om)$, we have that
		$g\left(\frac{|w_i(x)-w_i(y)|}{|x-y|^s}\right)$ is a bounded sequence in $L^{\overline{G}}(\frac{1}{|x-y|^N},\RR^N\times\RR^N)$ hence it has a weakly convergent subsequence. Thus we conclude that, up to a subsequence,
		$$
		g\left(\frac{|w_i(x)-w_i(y)|}{|x-y|^s}\right) \rightharpoonup g\left(\frac{|w(x)-w(y)|}{|x-y|^s}\right)
		$$ weakly in $L^{\overline{G}}(\frac{1}{|x-y|^N},\RR^N\times\RR^N)$. Now, since $\frac{|\varphi(x)-\varphi(y)|}{|x-y|^s}\in L^G(\frac{1}{|x-y|^N},\RR^N\times\RR^N)$,
		$$\I{\RR^N}\I{\RR^N} g\left(\frac{|w_i(x)-w_i(y)|}{|x-y|^s}\right)\frac{|\varphi(x)-\varphi(y)|}{|x-y|^{N+s}}dxdy\to 
		\I{\RR^N}\I{\RR^N} g\left(\frac{|w(x)-w(y)|}{|x-y|^s}\right)\frac{|\varphi(x)-\varphi(y)|}{|x-y|^{N+s}}dxdy$$
		Since the solution so obtained is in $u_n\in W^{s,G}_0(\Om)\cap L^\infty(\Om)$ and hence it is $C^{\alpha (n)}(\Om)$ where $\alpha(n)\in (0,1),\;\forall n\in\NN$ by Theorem 1.1 of Bonder et al \cite{Bo}.
	\end{proof}
	\begin{lemma}\label{lemma monotone}
		Assume $g$ to be convex on $(0,1)$.	The sequence of functions $\{u_n\}_n$, found in \cref{lemma approx problem} satisfies
		$$
		u_n(x)\leq u_{n+1}(x), \quad \mbox{for almost every } x\in\Om,
		$$
		and for any compact set $K\subseteq \Om$, there exists a constant $l=l(K)>0$ such that for any $n$, large enough,
		$$
		u_n(x)\geq l \quad \mbox{for almost every }x\in K.
		$$
	\end{lemma}
	\begin{proof}
		Set the notation $w_n(x)=(u_{n}(x)-u_{n+1}(x))^+$. Then we note that, for any $x\in\Om$, and $f_n(x)\leq f_{n+1}(x)$, 
		\begin{align*}
			\I{\Om} &\frac{f_n(x)}{(u_n(x)+\frac{1}{n})^{q(x)}}w_n(x)dx-\I{\Om}\frac{f_{n+1}(x)}{(u_{n+1}(x)+\frac{1}{n+1})^{q(x)}}w_n(x)dx\\
			&=\I{\Om}\left( \frac{f_n(x)}{(u_n(x)+\frac{1}{n})^{q(x)}}-\frac{f_{n+1}(x)}{(u_{n+1}(x)+\frac{1}{n+1})^{q(x)}} \right)w_n(x)dx\\
			&=\I{\Om}\left( \frac{f_n(x)}{(u_n(x)+\frac{1}{n})^{q(x)}}-\frac{f_{n+1}(x)}{(u_{n+1}(x)+\frac{1}{n+1})^{q(x)}} \right)(u_n(x)-u_{n+1}(x))^+dx\\
			&\leq \I{\{u_{n}(x)>u_{n+1}(x)\}}f_{n+1}(x)\left( \frac{(u_{n+1}(x)+\frac{1}{n+1})^{q(x)}-(u_n(x)+\frac{1}{n})^{q(x)}}{(u_n(x)+\frac{1}{n})^{q(x)}(u_{n+1}(x)+\frac{1}{n+1})^{q(x)}} \right)(u_{n}-u_{n+1})^+dx\\
			&\leq 0.
		\end{align*}
		Then the above calculation and \cref{eq approx problem} implies
		\begin{align*}
			\I{\RR^N}\I{\RR^N}&g\left(\frac{u_{n}(x)-u_{n}(y)}{|x-y|^s}\right)\frac{w_n(x)-w_n(y)}{|x-y|^{N+s}}dxdy\\
			&\leq \I{\RR^N}\I{\RR^N}g\left(\frac{u_{n+1}(x)-u_{n+1}(y)}{|x-y|^s}\right)\frac{w_n(x)-w_n(y)}{|x-y|^{N+s}}dxdy.
		\end{align*}
		Now \cite[Theorem~1.1]{BoSaVi} implies that both $u_n,u_{n+1}$ are H\"{o}lder continuous up to the boundary. So, we can apply \cref{comparison} to get $u_n\leq u_{n+1}$ a.e. on $\RR^N$. This concludes the proof of the first part.
		
		The second part follows from the continuity of $u_n$, and \cref{lemma smp}, which gives $u_n>0$ on $\Om$.
	\end{proof}

	\begin{proof}[\textbf{Proof of \cref{main 1}}]
		By \cref{lemma approx problem}, \cref{eq approx problem} has a weak solution $u_n$. Let $\varphi\in \test(\Om)$. We have
		\begin{equation}\label{eq 1}
			\I{\RR^N}\I{\RR^N}g\left(\frac{u_{n}(x)-u_{n}(y)}{|x-y|^s}\right)\frac{\varphi(x)-\varphi(y)}{|x-y|^{N+s}}dxdy
			=\I{\Om}\frac{f_n(x)\varphi(x)}{(u_n(x)+\frac{1}{n})^{q(x)}}dx.
		\end{equation}
		First, we claim:
		\begin{multline}\label{eq 2}
			\lim_{n\to\infty}\I{\RR^N}\I{\RR^N}g\left(\frac{u_{n}(x)-u_{n}(y)}{|x-y|^s}\right)\frac{\varphi(x)-\varphi(y)}{|x-y|^{N+s}}dxdy\\
			=
			\I{\RR^N}\I{\RR^N}g\left(\frac{u(x)-u(y)}{|x-y|^s}\right)\frac{\varphi(x)-\varphi(y)}{|x-y|^{N+s}}dxdy.
		\end{multline}
		\textbf{Proof of the claim:} 
		Set $\omega_\delta:=\Om\setminus\Om_\delta$. Then by \cref{lemma monotone}, there exists a constant $l>0$ such that $u_n\geq l>0$ on $\omega_\delta$. We get, using \cref{lemma embedding}, and choosing $\varphi=u_n$,
		\begin{align*}
			C(G) &\I{\RR^N}\I{\RR^N}G\left(\frac{|u_{n}(x)-u_{n}(y)|}{|x-y|^s}\right)\frac{dxdy}{|x-y|^{N}}\\
			&\leq \I{\RR^N}\I{\RR^N}g\left(\frac{u_{n}(x)-u_{n}(y)}{|x-y|^s}\right)\frac{u_n(x)-u_n(y)}{|x-y|^{N+s}}dxdy\\
			&= \I{\Om}\frac{f_n(x)u_n(x)}{(u_n(x)+\frac{1}{n})^{q(x)}}dx\\
			&=\I{\Om_\delta}\frac{f_n(x)u_n(x)}{(u_n(x)+\frac{1}{n})^{q(x)}}dx
			+\I{\omega_\delta}\frac{f_n(x)u_n(x)}{(u_n(x)+\frac{1}{n})^{q(x)}}dx\\
			&\leq \I{\Om_\delta\cap \{u_n\leq1\}}f_n(x)dx
			+\I{\Om_\delta\cap\{u_n>1\}}f_n(x)u_n(x)dx
			+\I{\omega_\delta}\frac{f_n(x)u_n(x)}{l^{q(x)}}dx\\
			&\leq \|f\|_{L^1(\Om)}+(1+\|l^{-q(\cdot)}\|_{L^\infty(\omega_\delta)})\|f\|_{L^{\overline{G_*}}(\Om)}\|u_n\|_{L^{G_*}(\Om)}\\
			&\leq \|f\|_{L^1(\Om)}+C_1\|u_n\|_{W^{s,G}_0(\Om)}.
		\end{align*}
		Assuming $\alpha:=\|u_n\|_{W^{s,G}_0(\Om)}>1$, we get, using \cref{lemma delta2},
		\begin{multline*}
			1=\I{\RR^N}\I{\RR^N}G\left(\frac{|u_{n}(x)-u_{n}(y)|}{\alpha|x-y|^s}\right)\frac{dxdy}{|x-y|^{N}}\\
			\leq \frac{1}{\alpha^{p^-}} \I{\RR^N}\I{\RR^N}G\left(\frac{|u_{n}(x)-u_{n}(y)|}{|x-y|^s}\right)\frac{dxdy}{|x-y|^{N}}\leq \frac{\|f\|_{L^1(\Om)}}{\alpha^{p^-}}+C_1\frac{1}{\alpha^{p^--1}}
		\end{multline*}
		This shows that $\|u_n\|_{W^{s,G}_0(\Om)}$ must be bounded. So $u_n\rightharpoonup u$ in $W^{s,G}_0$ weakly. By \cref{lemma embedding}, $u_n\to u$ strongly in $L^{1}(\Om)$, and hence $u_n\to u$ pointwise a.e. up to a subsequence.

		Now applying \cref{lemma delta2}
		\begin{equation*}
			\overline{G}(g(t))=\I{0}^{g(t)}\overline{g}(\tau)d\tau
			=\I{0}^t\overline{g}(g(\tau))g'(\tau)d\tau
			=\I{0}^t\tau g'(\tau)d\tau.
		\end{equation*}
		This implies
		\begin{equation}\label{eq 3}
			(p^--1)G(t)\leq \overline{G}(g(t)) \leq (P^+-1)G(t).
		\end{equation}
		This, along with \cref{lemma equivalent norm}, shows that the sequence of functions $(x,y)\mapsto g\left(\frac{u_{n}(x)-u_{n}(y)}{|x-y|^s}\right)$ is bounded in $L^{\overline{G}}\left(\RR^N\times\RR^N,\frac{dxdy}{|x-y|^N}\right)$. So it has a weakly convergent subsequence; without loss of generality, we assume it to be itself. It is easy to check that the the function $(x,y)\mapsto\frac{\varphi(x)-\varphi(y)}{|x-y|^s}$ is in $L^{G}\left(\RR^N\times\RR^N,\frac{dxdy}{|x-y|^N}\right)$.
		Hence \cref{eq 2} follows and the claim is true.

		Now, in order to complete the proof, taking into account \cref{eq 1}, we only need to show the convergence of the right-hand side of \cref{eq 1}. Note that 
		$$
		\left|\frac{f_n(x)\varphi(x)}{(u_n(x)+\frac{1}{n})^{q(x)}}\right|\leq |l^{-q(x)}f(x)\varphi(x)|\in L^1(\Om),
		$$
		where we get $l$ from applying \cref{lemma monotone} on $\mbox{supp}(\varphi)$. Therefore, we can apply DCT to get 
		$$
		\lim_{n\to\infty}\I{\Om}\frac{f_n(x)\varphi(x)}{(u_n(x)+\frac{1}{n})^{q(x)}}dx
		=\I{\Om}\frac{f(x)\varphi(x)}{u(x)^{q(x)}}dx.
		$$
		
		Hence the proof is complete.
		
	\end{proof}
	
	\begin{lemma}\label{lemma estimate}
		For any $a,b\in\RR$, we have 
		$$
		|g(a)-g(b)| 
		\leq C \frac{|a-b|g(|a|+|b|)}{|a|+|b|}
		\leq C g(|a|+|b|).
		$$
	\end{lemma}
	\begin{proof}
		\begin{equation*}
			g(b)-g(a)=\I{0}^1g'(a+(b-a)t)(b-a)dt.
		\end{equation*}
		Now since $g'$ is increasing one has for $t\in(0,1)$, $|a+(b-a)t|\leq\left||a|+|b|\right|$. So we get $$|g(a)-g(b)|\leq |a-b|g'(|a|+|b|)$$ The results now follow trivially using the hypothesis on $g$.
	\end{proof}
	\begin{lemma}\label{lemma MVT}
		Let $\Phi:(0,\infty)\to(0,\infty)$ be a strictly convex, $C^1$-function such that $\Phi'$ is increasing and there exists $\theta_1,\theta_2\geq 0$ such that $\theta_1\frac{\Phi(x)}{x}\leq \Phi'(x)\leq\theta_2\frac{\Phi(x)}{x}$. For $x,y\in \RR$ and $\varepsilon>0$, define $S^x_\varepsilon:=\{x\geq\varepsilon\}\cap\{y\geq0\}$, and $S^y_\varepsilon:=\{x\geq0\}\cap\{y\geq\varepsilon\}$. Then for $(x,y)\in S^x_\varepsilon\cup S^y_\varepsilon$, 
		$$
		|\Phi(x)-\Phi(y)|\geq C\Phi'(\epsilon)|x-y|\;\mbox{with}\;C:=\max(\theta_1,1).
		$$
		\begin{proof}
			By symmetry, without loss of generality, we can assume $x> y$. Now for some $\lambda\in(y,x)$, we have $\Phi(x)-\Phi(y)=\Phi'(\lambda)(x-y)$. If we assume $x\geq y\geq \varepsilon>0$, then we have 
			$$
			|\Phi(x)-\Phi(y)|
			\geq \Phi'(\lambda)|x-y|
			\geq \Phi'(\varepsilon)|x-y|
			$$
			For, $0\leq y<\varepsilon\leq x$, then by strict convexity of $\Phi$, we get $$\frac{\Phi(x)-\Phi(y)}{x-y}\geq \frac{\Phi(x)}{x}\geq \theta_1\Phi'(x)\geq \theta_1\Phi'(\varepsilon)$$ 
			thus concluding the assertion.
		\end{proof}
	\end{lemma}

	\begin{lemma}\label{lemma bound}
		Let $\Phi,\ H,\ f,\ q$ be as in \cref{main 2}, $u_n$ be as in \cref{lemma approx problem}. Then there is a constant $C>0$, independent of $n$ such that $\|\Phi(u_n)\|_{W^{s,G}_0(\Om)}, \ \|\Phi(u)\|_{W^{s,G}_0(\Om)}\leq C$, where $u$ is the pointwise limit of $u_n$.
	\end{lemma}
	\begin{proof}
		We have, for $t>0$,
		$$
		\Phi(t):=\I{0}^tG^{-1}\left(G(1)\tau^{q^*-1}\right)d\tau,
		$$
		that is
		$$
		\Phi'(t):=G^{-1}\left(G(1)t^{q^*-1}\right),
		$$ 
		which gives, applying the fact that $\Phi'(t)$ is increasing and hence $\Phi(t)\leq t\Phi'(t)$,
		\begin{equation}\label{eq 7}
			g\left(\Phi'(t)\right)\Phi(t)
			=\frac{\Phi'(t)g\left(\Phi'(t)\right)}{G\left(\Phi'(t)\right)}\frac{G\left(\Phi'(t)\right)\Phi(t)}{\Phi'(t)}
			\leq p^+G(1)t^{q^*-1}\frac{\Phi(t)}{\Phi'(t)}
			\leq p^+G(1)t^{q^*}.
		\end{equation}
		Using \cref{eq 7,lemma convex}, and the fact $q^*>1$ we have
		\begin{align}\label{eq 4}
			\nonumber \I{\RR^N}\I{\RR^N}&G\left(\frac{|\Phi(u_n(x))-\Phi(u_n(y))|}{|x-y|^s}\right)\frac{dxdy}{|x-y|^N}\\
			\nonumber &\leq C \I{\Om}\frac{f_n(x)}{(u_n(x)+\frac{1}{n})^{q(x)}}g(\Phi'(u_n(x)))\Phi(u_n(x)) dx\\
			\nonumber &= C \left(\I{\Om_\delta,\ u_n<1}+ \I{\Om_\delta,\ u_n\geq1}+\I{\omega_\delta,\ u_n<1}+\I{\omega_\delta,\ u_n\geq1}\right)\frac{f_n(x)}{(u_n(x)+\frac{1}{n})^{q(x)}}g(\Phi'(u_n(x)))\Phi(u_n(x)) \\
			&\leq C\I{\Om\cap\{u_n<1\}}f_n(x)+ C\I{\Om\cap\{u_n\geq 1\}}f_n(x)u_n(x)^{q^*}.
		\end{align}
		Set $r:=\frac{p^-}{p^-+q^*-1}$. We have, for large enough $t_0$, and for any $t>t_0$,
		\begin{align}\label{eq-r-bound}
			\nonumber t^\frac{1}{r}
			&=\frac{1}{r}\I{0}^{1}\tau^{\frac{1}{r}-1}d\tau
			+ \frac{1}{r}\I{1}^{t}\tau^{\frac{1}{r}-1}d\tau
			\leq \frac{2}{r}\I{1}^{t}\tau^{\frac{1}{r}-1}G^{-1}(G(1))d\tau\\
			&\leq \frac{2}{r}\I{1}^{t}G^{-1}(G(1)\tau^{\frac{p^-(1-r)}{r}})d\tau
			\leq \frac{2}{r}\I{0}^{t}G^{-1}(G(1)\tau^{\frac{p^-(1-r)}{r}})d\tau
			= \frac{2}{r}\Phi(t).
		\end{align} 
		Applying \cref{eq-r-bound} on \cref{eq 4} and then using H\"older's inequality, and finally the fact that $|f_n|\leq |f|$, we get
		\begin{align}\label{eq 6}
		\nonumber	\I{\RR^N}\I{\RR^N}&G\left(\frac{|\Phi(u_n(x))-\Phi(u_n(y))|}{|x-y|^s}\right)\frac{dxdy}{|x-y|^N}\\
			\nonumber &\leq C\I{\Om\cap\{u_n<1\}}f_n(x)+ C\I{\Om\cap\{u_n\geq1\}}f_n(x)\Phi(u_n(x))^{rq^*}\\
			\nonumber &\leq C\|f_n\|_{L^1(\Om)}+C\|f_n\|_{L^{\overline{H}}(\Om)}\|\Phi(u_n)^{rq^*}\|_{L^{H}(\Om)}\\
			 &\leq C\|f\|_{L^1(\Om)}+C\|f\|_{L^{\overline{H}}(\Om)}\|\Phi(u_n)^{rq^*}\|_{L^{H}(\Om)}.
		\end{align}
		Observe that 
		\begin{align*}
			\left\| \Phi^{rq^*}(u_n) \right\|_{L^H(\Om)}
			&= \inf \left\{ \lambda>0\ \Big|\ \I{\Om} H\left(\frac{\Phi(u_n)^{rq^*}}{\lambda}\right) \leq 1\right\}\\
			&= \inf \left\{ \lambda^{rq^*}>0\ \Big|\ \I{\Om} H\left(\frac{\Phi(u_n)^{rq^*}}{\lambda^{rq^*}}\right) \leq 1\right\}\\
			&= \left(\inf \left\{ \lambda>0\ \Big|\ \I{\Om} H\left(\frac{\Phi(u_n)^{rq^*}}{\lambda^{rq^*}}\right) \right\}\right)^{rq^*}\\
			&=\left(\inf \left\{ \lambda>0\ \Big|\ \I{\Om} G_*\left(\frac{\Phi(u_n)}{\lambda}\right) \right\}\right)^{rq^*}
			=\left\| \Phi(u_n) \right\|_{L^{G_*}(\Om)}^{rq^*},
		\end{align*}
		to see the last line recall that $G_*(t):=H\left(t^{rq^*}\right)$. Combining this with \cref{eq 6} gives
		$$\I{\RR^N}\I{\RR^N}G\left(\frac{|\Phi(u_n(x))-\Phi(u_n(y))|}{|x-y|^s}\right)\frac{dxdy}{|x-y|^N}
		\leq C\|f\|_{L^1(\Om)}+C\|f\|_{L^{\overline{H}}(\Om)}\|\Phi(u_n)\|_{L^{G_*}(\Om)}^{rq^*}.$$
		From \cref{lemma embedding}, we can write
		$$\I{\RR^N}\I{\RR^N}G\left(\frac{|\Phi(u_n(x))-\Phi(u_n(y))|}{|x-y|^s}\right)\frac{dxdy}{|x-y|^N}
		\leq C\|f\|_{L^1(\Om)}+C\|f\|_{L^{\overline{H}}(\Om)}\|\Phi(u_n)\|_{W^{s,G}_0(\Om)}^{rq^*}.$$    
		When $ \|\Phi(u_n)\|_{W^{s,G}_0(\Om)}>t_0$, using \cref{lemma equivalent norm}, we get
		$$
		\|\Phi(u_n)\|_{W^{s,G}_0(\Om)}^{p^-}
		\leq C\|f\|_{L^1(\Om)}+C\|f\|_{L^{\overline{H}}(\Om)}\|\Phi(u_n)\|_{W^{s,G}_0(\Om)}^{rq^*}.
		$$
		From the hypothesis, we have, ${rq^*}<p^-$. This implies that the norm $\|\Phi(u_n)\|_{W^{s,G}_0(\Om)}$ cannot increase arbitrarily.
		So, there exists a constant $C>0$, independent of $n$, such that $ \|\Phi(u_n)\|_{W^{s,G}_0(\Om)}\leq C$.
		
		By \cref{lemma monotone}, $u_n$ is a monotone increasing sequence. So, we can define $u$ as the pointwise limit of $u_n$. Direct application of Fatou's lemma and \cref{lemma equivalent norm} implies that $ \|\Phi(u)\|_{W^{s,G}_0(\Om)}\leq C$.
	\end{proof}

	\begin{proof}[\textbf{Proof of \cref{main 2}}]
		
		By \cref{lemma monotone}, $u_n$ is a monotone increasing sequence. So, we can define $u$ as the pointwise limit of $u_n$. Next, we show that this $u$ is the required solution.

		We know from \cref{lemma approx problem} that there are $u_n$ which satisfy
		\begin{equation*}
			\I{\RR^N}\I{\RR^N}g\left(\frac{u_n(x)-u_n(y)}{|x-y|^s}\right)\frac{\varphi(x)-\varphi(y)}{|x-y|^{N+s}}dxdy
			=\I{\Om}\frac{f_n(x)\phi(x)}{\left(u_n(x)+\frac{1}{n}\right)^{q(x)}}dx.
		\end{equation*}
		Note that, on $\mbox{supp}(\phi)$, as $f\in L^1(\Om)$,
		$$\left| \frac{f_n(x)\phi(x)}{\left(u_n(x)+\frac{1}{n}\right)^{q(x)}} \right|\leq \|l^{-q(\cdot)}\|_{L^\infty} |f||\phi|\in L^1.$$
		Hence, by dominated convergence theorem, we get 
		$$
		\lim_{n\to \infty}\I{\Om}\frac{f_n(x)\phi(x)}{\left(u_n(x)+\frac{1}{n}\right)^{q(x)}}=\I{\Om}\frac{f(x)\phi(x)}{u(x)^{q(x)}}.
		$$
		So, we need to show that 
		$$\lim_{n\to \infty}\I{\RR^N}\I{\RR^N}g\left(\frac{u_n(x)-u_n(y)}{|x-y|^s}\right)\frac{\varphi(x)-\varphi(y)}{|x-y|^{N+s}}dxdy
		=\I{\RR^N}\I{\RR^N}g\left(\frac{u(x)-u(y)}{|x-y|^s}\right)\frac{\varphi(x)-\varphi(y)}{|x-y|^{N+s}}dxdy.$$
		We have, $\Phi(u)\in W^{s,G}_0(\Om)$ and by \cref{poincare}, it follows that $\Phi(u)\in L^G(\Om)$. Comparing integrals, where $u>1$, it follows that $u\in L^G(\Om)$. We see, using \cref{lemma estimate}, 
		\begin{align*}
		&\left| \I{\RR^N}\I{\RR^N}g\left(\frac{u_n(x)-u_n(y)}{|x-y|^s}\right)\frac{\varphi(x)-\varphi(y)}{|x-y|^{N+s}}dxdy
		-\I{\RR^N}\I{\RR^N}g\left(\frac{u(x)-u(y)}{|x-y|^s}\right)\frac{\varphi(x)-\varphi(y)}{|x-y|^{N+s}}dxdy\right| \\
			&=  \I{\RR^N}\I{\RR^N}\left|g\left(\frac{u_n(x)-u_n(y)}{|x-y|^s}\right)-g\left(\frac{u(x)-u(y)}{|x-y|^s}\right)\right|\frac{|\varphi(x)-\varphi(y)|}{|x-y|^{N+s}}dxdy\\
			&\leq C \I{\RR^N}\I{\RR^N}g\left(\frac{|u_n(x)-u_n(y)|+|u(x)-u(y)|}{|x-y|^s}\right)\frac{|\varphi(x)-\varphi(y)|}{|x-y|^{N+s}}dxdy\\
			&=C \I{\RR^N}\I{\RR^N}I_n\quad \mbox{(assume)}
		\end{align*}
		The proof will be complete if we can show that $\I{\RR^N}\I{\RR^N}I_n\to 0$. To do this, first, set
		
		$$
		\mathcal{S_\phi}:=\mbox{supp}\phi,\quad \mbox{and}\quad \mathcal{Q}_\phi:=(\RR^N\times\RR^N)\setminus(\mathcal{S_\phi}^c\times\mathcal{S_\phi}^c).
		$$
		Now using H\"older's inequality with respect to the measure $\frac{dxdy}{|x-y|^N}$, we get for any compact set $K\subseteq \RR^N\times\RR^N$,
		\begin{align*}
			\II{\RR^{2N}\setminus K}I_n
			&=\II{\mathcal{Q}_\phi\setminus K}I_n\\
			&\leq C \left\|g\left(\frac{|u_n(x)-u_n(y)|+|u(x)-u(y)|}{|x-y|^s}\right)\right\|_{L^{\overline{G}}(\mathcal{Q}_\phi\setminus K,\frac{dxdy}{|x-y|^N})}
			\left\|\frac{|\varphi(x)-\varphi(y)|}{|x-y|^s}\right\|_{L^G(\mathcal{Q}_\phi\setminus K,\frac{dxdy}{|x-y|^N})}.
		\end{align*}
		Now, if $\left\|g\left(\frac{|u_n(x)-u_n(y)|+|u(x)-u(y)|}{|x-y|^s}\right)\right\|_{L^{\overline{G}}(\mathcal{Q}_\phi\setminus K,\frac{dxdy}{|x-y|^N})}\leq1$, we get
		
		$$
		\II{\RR^{2N}\setminus K}I_n\leq C \left\|\frac{|\varphi(x)-\varphi(y)|}{|x-y|^s}\right\|_{L^G(\mathcal{Q}_\phi\setminus K,\frac{dxdy}{|x-y|^N})}.
		$$
		
		Otherwise, we apply \cref{lemma equivalent norm,eq 3} to get
		\begin{align*}
			&\II{\RR^{2N}\setminus K}I_n\\
			&\leq C \left(\II{\mathcal{Q}_\phi\setminus K}\overline{G}\left(g\left(\frac{|u_n(x)-u_n(y)|+|u(x)-u(y)|}{|x-y|^s}\right)\right)\frac{dxdy}{|x-y|^N}\right)^\frac{1}{p^-}
			\left\|\frac{|\varphi(x)-\varphi(y)|}{|x-y|^s}\right\|_{L^G(\mathcal{Q}_\phi\setminus K,\frac{dxdy}{|x-y|^N})}\\
			&\leq C \left(\II{\mathcal{Q}_\phi\setminus K}G\left(\frac{|u_n(x)-u_n(y)|+|u(x)-u(y)|}{|x-y|^s}\right)\frac{dxdy}{|x-y|^N}\right)^\frac{1}{p^-}
			\left\|\frac{|\varphi(x)-\varphi(y)|}{|x-y|^s}\right\|_{L^G(\mathcal{Q}_\phi\setminus K,\frac{dxdy}{|x-y|^N})}\\
			&\leq C \left[\II{\mathcal{Q}_\phi\setminus K}G\left(\frac{|u_n(x)-u_n(y)|}{|x-y|^s}\right)\frac{dxdy}{|x-y|^N}+\II{\mathcal{Q}_\phi\setminus K}G\left(\frac{|u(x)-u(y)|}{|x-y|^s}\right)\frac{dxdy}{|x-y|^N}\right]^\frac{1}{p^-}\\
			&\hspace{200pt}\hfill\times\left\|\frac{|\varphi(x)-\varphi(y)|}{|x-y|^s}\right\|_{L^G(\mathcal{Q}_\phi\setminus K,\frac{dxdy}{|x-y|^N})} .
		\end{align*}
		By \cref{lemma monotone}, there exists $l=l(\mathcal{S}_\phi)>0$ such that for $n$ large enough, $u_n(x)>l$. We now apply \cref{lemma MVT} on the two integrands of the last line to get
		\begin{align*}
			\II{\RR^{2N}\setminus K}I_n
			&\leq C \left[\II{\mathcal{Q}_\phi\setminus K}G\left(\frac{|\Phi(u_n)(x)-\Phi(u_n)(y)|}{|x-y|^s}\right)\frac{dxdy}{|x-y|^N}\right.\\
			&\left. +\II{\mathcal{Q}_\phi\setminus K}G\left(\frac{|\Phi(u(x))-\Phi(u(y))|}{|x-y|^s}\right)\frac{dxdy}{|x-y|^N}\right]^\frac{1}{p^+}\left\|\frac{|\varphi(x)-\varphi(y)|}{|x-y|^s}\right\|_{L^G(\mathcal{Q}_\phi\setminus K,\frac{dxdy}{|x-y|^N})} .
		\end{align*}

		By \cref{lemma equivalent norm,lemma bound}, it is clear that
		\begin{equation*}
			\II{\RR^{2N}\setminus K}I_n\\
			\leq C \left\|\frac{|\varphi(x)-\varphi(y)|}{|x-y|^s}\right\|_{L^G(\mathcal{Q}_\phi\setminus K,\frac{dxdy}{|x-y|^N})} .
		\end{equation*}
		Since $\phi\in C_c^\infty(\Om)$, for a fixed $\varepsilon>0$, there exists $K=K(\varepsilon)$ such that 
		
		\begin{equation*}
			\II{\RR^{2N}\setminus K}I_n< \frac{\varepsilon}{2}.
		\end{equation*}
		
		We, now have to estimate $\II{K}I_n$. For this, we use Vitali's convergence theorem. Let $E\subseteq K$. Arguing as above, we can get
		
		\begin{equation*}
			\II{E}I_n\\
			\leq C \left\|\frac{|\varphi(x)-\varphi(y)|}{|x-y|^s}\right\|_{L^G(E,\frac{dxdy}{|x-y|^N})} .
		\end{equation*}
		This shows that the integrand in LHS is uniformly integrable, that is $\II{E}I_n\to 0$ as $\mathcal{L}^N(E)\to 0$.
		Applying Vitali's convergence theorem, we get for large enough $n$, $\II{E}I_n<\frac{\varepsilon}{2}$.
		So, from \cref{eq 6}, we get $\I{\RR^N}\I{\RR^N}I_n\to 0$ as $n\to \infty$, hence the proof follows.
	\end{proof}

	\begin{proof}[\textbf{Proof of \cref{main 3}}]
		Let $u$ be a solution of \cref{eq problem} obtained through \cref{main 1,main 2}. Then $u$ is pointwise limit of a sequence of solutions, $u_n$, of \cref{eq approx problem}. Also, by \cref{lemma monotone}, there exists $l(K)>0$ for any compact set $K\subseteq \Om$ such that
		$$
		u(x)\geq l(K)>0\quad \mbox{for almost all } x\in K.
		$$
		This implies that there exists some $C_K>0$ such that $u^{-q(x)}(x)\leq C_K$ for all $x\in K$. Fix $x_0\in \Om$ and $r>0$ such that $B:=B(x_0,r)\subset \overline{B(x_0,r)}\subset \Om$. Again, since $u$ is a weak solution of \cref{eq problem}, this implies that for any $\varphi\in\test(B(x_0,r))$, where , with $\varphi\geq 0$,
		\begin{multline}\label{eq 5}
			\I{\RR^N}\I{\RR^N}g\left(\frac{u(x)-u(y)}{|x-y|^s}\right)\frac{(\varphi(x)-\varphi(y))}{|x-y|^{N+s}}dxdy
			=\I{B}f(x)u(x)^{-q(x)}\phi(x)dx\\
			\leq C_B\I{B}f(x)\phi(x)dx
			=  \I{\RR^N}\I{\RR^N}g\left(\frac{v(x)-v(y)}{|x-y|^s}\right)\frac{(\varphi(x)-\varphi(y))}{|x-y|^{N+s}}dxdy,
		\end{multline}
		where $v\in W^{s,G}(B)\cap  L^\infty(B)$, is a solution to the problem
		\begin{equation*}
			\begin{cases}
				(-\D_g)^sv=C_Bf, \quad \mbox{ in } B,\\
				v>0, \mbox{ in } B,\\
				v=0, \mbox{ in } \RR^N\setminus B
			\end{cases}
		\end{equation*}
		obtained through \cref{laplace equation}. By using \cref{comparison}, we can conclude that $u\leq v$ in $B$ if $u$ is continuous on $\RR^N$. That is $u\in L^\infty_{loc} (\Om)$ provided $u$ is continuous on $\RR^N$.
		
		Again, since we have, from \cref{eq 5},
		$$
		\I{\RR^N}\I{\RR^N}g\left(\frac{u(x)-u(y)}{|x-y|^s}\right)\frac{(\varphi(x)-\varphi(y))}{|x-y|^{N+s}}dxdy
		\leq C\I{B}\phi(x)dx,
		$$
		defining the sets 
		\begin{align*}
			U_0&:=\left\{ (x,y)\in\RR^N\times\RR^N\ \Big|\ \frac{|u(x)-u(y)|}{|x-y|^s}\geq 1 \right\},\\
			U_j&:=\left\{ (x,y)\in\RR^N\times\RR^N\ \Big|\ \frac{1}{j+1}\leq\frac{|u(x)-u(y)|}{|x-y|^s}<\frac{1}{j} \right\}\quad \mbox{ for } j\geq 1,
		\end{align*}
		we get from \cref{lemma delta2} that
		\begin{align*}
			C\I{B}\phi(x)dx
			&\geq \I{\RR^N}\I{\RR^N}g\left(\frac{u(x)-u(y)}{|x-y|^s}\right)\frac{(\varphi(x)-\varphi(y))}{|x-y|^{N+s}}dxdy\\
			&= \I{\RR^N}\I{\RR^N}g\left(\frac{u(x)-u(y)}{|x-y|^s}\right)\frac{u(x)-u(y)}{|x-y|^s}\frac{(\varphi(x)-\varphi(y))}{(u(x)-u(y))|x-y|^N}dxdy\\
			&\geq p^-\I{\RR^N}\I{\RR^N}G\left(\frac{|u(x)-u(y)|}{|x-y|^s}\right)\frac{(\varphi(x)-\varphi(y))}{(u(x)-u(y))|x-y|^N}dxdy\\
			&= p^-\sum_{j=0}^\infty j^{p^+}G(\frac{1}{j+1})\II{U_j} \frac{|u(x)-u(y)|^{p^+-2}(u(x)-u(y))(\phi(x)-\phi(y))}{|x-y|^{N+sp^+}}dxdy\\
			&\geq p^-\sum_{j=0}^\infty \frac{j^{p^+}}{(j+1)^{p^+}}G(1)\II{U_j} \frac{|u(x)-u(y)|^{p^+-2}(u(x)-u(y))(\phi(x)-\phi(y))}{|x-y|^{N+sp^+}}dxdy\\
			&\geq C\I{\RR^N}\I{\RR^N} \frac{|u(x)-u(y)|^{p^+-2}(u(x)-u(y))(\phi(x)-\phi(y))}{|x-y|^{N+sp^+}}dxdy
			.
		\end{align*}
		We can now apply Corollary~5.5 of \cite{AnSuSq} to conclude that there is some $\alpha\in(0,1)$ such that $u\in C^\alpha(B)$.
		This completes the proof.
	\end{proof}
	\bigskip
	
	\section*{Acknowledgement}
	\noindent The first author was funded by MATRICS (DST, INDIA) project MTR/2020/000594.\\
	The second and the third author were funded by Academy of Finland grant: Geometrinen Analyysi (21000046081).

	\section*{Author Contribution}
	All the authors have contributed equally to the article.

	\section*{Conflict of Interest}
	The authors have no competing interests to declare that are relevant to the content of this article.
	\bigskip

	\bibliographystyle{alpha}
	\bibliography{bibliography_Orlicz.bib}
\end{document}